\documentclass{amsart}

\usepackage{amsmath,amssymb,amsthm,verbatim,mathrsfs}
\usepackage{hyperref}
\usepackage{xcolor}
\usepackage{bookmark}

\newtheorem{thm}{Theorem}[section]
\newtheorem{lem}[thm]{Lemma}
\newtheorem{cor}[thm]{Corollary}

\theoremstyle{definition}

\newtheorem{rk}[thm]{Remark}

\newcommand{\bC}{\mathbf{C}}

\newcommand{\p}{\partial}
\newcommand{\R}{\mathbb{R}}
\newcommand{\reg}{\mathrm{reg}}
\newcommand{\sing}{\mathrm{sing}}

\title{Harmonic functions on minimal submanifold with cylindrical cone}

\author{Yu Wang}
\address{School of Mathematical Sciences, Fudan University, Shanghai 200433, China}
\email{23110180040@m.fudan.edu.cn}

\begin{document}

\begin{abstract}
In this paper, we study the relationship between the dimension of linear space of harmonic function with growth bounded by a fixed-degree polynomial on a minimal submanifold in Euclidean space and that on its one cylindrical tangent cone at $\infty$ by using the method in \cite{Hua19,Hua23}. Specifically, if this degree does not coincide with any growth of polynomial-growth harmonic functions on the cone, then the corresponding dimensions are equal.
\end{abstract}

\maketitle

\section{Introduction}

Harmonic functions with polynomial growth on Riemannian manifold plays an important role in differential geometry. A famous conjecture proposed by Yau, the dimension of space of such functions with a fixed rate on complete Riemannian manifold with nonnegative Ricci curvature is finite, was solved completely by Colding and Minicozzi in \cite{CM97}. They also showed that such result holds for minimal submanifold $M\subset\R^N$ with Euclidean volume growth (i.e. the density of $M$ at infty $\Theta_M(\infty)$ is finite) in \cite{CM98}. In fact, for fixed $d\ge0,$ let $\mathcal{H}_d(M)$ be the linear space  of polynomial growth harmonic function consists of smooth harmonic functions $u(x)$ on $M$ such that there exists some positive constant $C$ independent of $x$ with
$$|u(x)|\le C(1+|x|^d) \text{ for all } x\in M$$
where $|x|$ denote the Euclidean distance from $x$ to origin, then for any $d>1$ there is constant $C=C(n)$ such that 
\begin{equation}
\dim\mathcal{H}_d(M)\le C\Theta_M(\infty)d^{n-1} \tag{$\star$} \label{eq01}
\end{equation}
For convenience, we always denote by $h_d(M)=\dim\mathcal{H}_d(M).$ 

It is well know that submanfiold in Euclidean space is minimal if and only if all coordinate functions restricted on submanifold is harmonic. So the coordinate functions are in $\mathcal{H}_1(M).$ The exponent of $d$ in (\ref{eq01}) is sharp, but the codimension estimate from (\ref{eq01}) is not optimum because of $M$ should be a plane when $\Theta_M(\infty)$ closes to one by Allard's regularity theorem. 

For Riemannian manifold $M$ with nonnegative Ricci curvature with maximal volume growth and unique tangent cone $\bC$ at $\infty,$ Huang \cite{Hua19,Hua23} showed that $h_d(M)=h_d(\bC)$ for any positive $d$ not equal to the order of harmonic functions with polynomial growth on $\bC.$ Here, the polynomial growth of function is defined with respect to Riemannian metric of $M.$ Colding and Minicozzi \cite{CM20} considered the space of polynomial growth caloric functions on the space-track of an ancient $n$ dimension mean curvature flow in $\R^n$ with finite entropy  and obtain a sharp estimate of codimension under an addition assumption, one tangent flow at $-\infty$ is a cylinder self shrinker.

Now we state main theorem of this paper. 

\begin{thm}\label{thm11}
Let $M^n$ be a smooth embedded minimal submanifold in $\R^N$. Suppose $\bC=\bC_0\times\R^{n-k}$ is one tangent cone of $M$ at $\infty$ with multiplicity one where $\bC_0=\bC_0(\Sigma)$ is a regular minimal hypercone with dimension $k.$
Then for any positive number $d\not\in\mathscr{D}(\bC),$ we have
\begin{equation*}
h_d(M)= h_d(\bC).
\end{equation*}
\end{thm}

Throughout this paper a function on $\bC$ we means that defined on $\reg\bC.$ Here $\mathscr{D}(\bC)=\{q_i\}_{i\ge 0}$ is set of all orders of polynomial growth harmonic functions on $\bC$ which can be seen in section 3. In particular, if denote by $\lambda_1$ the first non-zero eigenvalue of laplacian operator on $\Sigma,$ then 
$$q_1=-\frac{k-2}{2}+\sqrt{(\frac{k-2}{2})^2+\lambda_1}.$$
Since every minimal hypercone in $\R^{k+1}$ ($k+1\le 7$) is a hyperplane, we may assume $k>6.$ When $\bC_0$ is a quadratic cone, $\lambda_1=k-1$ with multiplicity $k+1,$ then we have the following corollary.

\begin{cor}
Under the hypothesis of Theorem \ref{thm11}, if $\bC_0$ is a quadratic cone, then $M$ is a hypersurface. Moreover,
\begin{itemize}
\item[(1)] for $k=n,$ if in addition $\bC_0$ is minimizing, then by \cite{SS86},
 up to translation and rotation, $M=H(\lambda)$ for some $\lambda\not=0$ where $H(\lambda)$ is the Hardt-Simon foliation;
\item[(2)] for $k<n,$ if in addition $\bC_0$ is strictly minimizing and strictly stable and $M$ lying to one side of $\bC$ and , then by \cite{ES24}, $M=H(\lambda)\times\R^l$ for some $\lambda\not=0.$ 
\end{itemize}
\end{cor}

If $M$ converges to $\bC$ with multiplicity great than one, Theorem 1.1 is failed. For example, let $M=\{|x_{n+1}|=1\}$ be the union of 2 parallel planes in $\R^{n+1},$ then $P=\{x_{n+1}=0\}$ with multiplicity $2$ is the tangent cone of $M$ at infinity. Obviously, $h_d(M)=2h_d(P)> h_d(P).$ One may conjecture that when $\bC$ is one tangent cone of $M$ at infinity with multiplicity $m>1,$ then $h_d(M)=mh_d(\bC).$

The article is arranged by the following. In section 2, we establish notations and list some lemmas that will be frequently used in the subsequent sections. In section 3, we discuss harmonic functions with polynomial growth on cylindrical cone $\bC.$ In section 4, we obtain a three-circle type theorem as \cite{Xu16} and follow the idea in \cite{Hua19,Hua23} to give a proof of Theorem \ref{thm11}.

\section*{Acknowledgements} The author would like to thank his advisor Prof. Qi Ding, for constant support and many helpful discussions.

\section{Preliminary}
Throughout this paper, $\bC=\bC_0\times\R^{n-k}\subset\R^N$ is a cylindrical cone where $\bC_0=\bC_0(\Sigma)$ is a $k$ dimension regular minimal cone with $\sing\bC_0=\{0\}$ and  $M^n\subset\R^N$ is a smooth embedded minimal submanifold such that there exists $R_i\to\infty$ so that the rescaled submanifold $M_i=R_i^{-1}M$ converges to $\bC$ as varifold with multiplicity one when $i\to\infty.$ 

The constant $C$ may change from line to line.

Let $B_s$ be the open ball in $R^N$ centered in origin with radius $s.$ For any $n$ dimension subset $U\subset\R^N$ and $u,v\in L^2_{loc}(U),$ set
\begin{equation*}
I^{(U)}_u(s)=s^{-n}\int_{U\cap B_{s}}u^2d\mathcal{H}^n,J^{(U)}_s(u,v)=s^{-n}\int_{U\cap B_s}uvd\mathcal{H}^n.
\end{equation*}
For the sake of convenience, hereafter we will consistently omit the superscripts on $I$ and $J,$ as well as the differentials  $d\mathcal{H}^n$ in integrals. 

Take $u\in\mathcal{H}_d(M),$ then $I_u(s)\le C(1+s^{2d}).$ From the monotonicity formula about subharmonic functions, it infers that $I_u(s)$ is non-decreasing with respect to $s.$ For non-decreasing functions with polynomial growth, an analogous Harnack inequality has been established by Colding and Minicozzi. The following lemma is a simplified version.

\begin{lem}[\cite{CM97}]\label{lem21}
Given any $\Omega>1$ and $\delta>0,$, there exists a sequence of integers $m_i\to\infty$ such that 
$$I_{u}(\Omega^{m_i+1})\le \Omega^{2d+\delta}I_{u}(\Omega^{m_i}).$$
\end{lem}

\cite{Guo24} extends the Rellich compactness theorem for Sobolev spaces on varifolds to a sequence of converging multiplicity one stationary varifolds. In our case, we have

\begin{lem}[\cite{Guo24}]\label{lem22}
Let $U$ be an open subset in $\R^N$ and $\tilde{M}_i=M_i\cap U,\tilde{\bC}=\bC\cap U.$ Suppose there is $W\subset U$ open so that $\overline{W}$ is compact and $\sup_{i}\Vert f_i\Vert_{W^{1,p}(\tilde{M}_i\cap W)}\le C$ then there exists a subsequence still denoted by $f_i$ and $f\in L^{q}(\tilde{\bC}\cap W)$ so that $f_i\to f$ in $L^q$ for $1\le q<\frac{np}{n-p}.$
\end{lem}

\section{Harmonic functions on cylindrical cone}
Take $\{\varphi_i\}_{i\ge0}$ be an $L^2$-ON basis on $\Sigma$ such that $\Delta_{\Sigma}\varphi_i+\lambda_i\varphi_i=0$ where
$$0=\lambda_0<\lambda_1\le\lambda_2\le\cdots\le\lambda_i\le\cdots,\lambda_i\to\infty \text{ as } i\to\infty$$
are eigenvalues of Laplacian operator on $\Sigma.$ By spectral theory, every $L^2$ function on $\Sigma$ can be expanded by $\{\varphi_i\}_{i\ge0}.$

For $(x,y)\in\bC,$ denote $x=r\theta,\rho\omega=(r,y)$ where $r=|x|,\rho=\sqrt{r^2+y^2}.$ A smooth function $v$ on $\bC\cap B_1$ with $\int_{\bC\cap B_1}r^{-2}v^2<\infty$ can be viewed as a $L^2$ function on $\Sigma$ if fixed $r,y,$ so we can write
$$v(x,y)=\sum_{i\ge0}v_i(r,y)\varphi_i(\theta)$$
where $v_i(r,y)=\int_{\Sigma}v(r\theta,y)\varphi_i(\theta),$ and then $v$ is harmonic iff for every $i\ge0$
$$\p_{rr}v_i+\frac{k-1}{r}\p_rv_i-\frac{\lambda_i}{r^2}v_i+\Delta_yv_i=0$$ 
on half ball $B^+=\{(r,y)\in\R^{n-k+1}:r>0, r^2+y^2<1\}.$ Set
\begin{equation*}
\mu_i=-\frac{k-2}{2}+\sqrt{(\frac{k-2}{2})^2+\lambda_i},
\end{equation*}
by direct computation, $\tilde{v}_i:=r^{-\mu_i}v_i$ such that the following equation
\begin{equation*}
r^{-1-\beta_i}\p_r(r^{1+\beta_i}\p_r\tilde{v}_i)+\Delta_y\tilde{v}_i=0,
\end{equation*}
and $\int_{B^+}r^{-2}\tilde{v}_i^2r^{1+\beta_i}<\infty$ where $\beta_i=2\mu_i+k-2.$ Such function $\tilde{v}_i$ is called $\beta_i$-harmonic which introduced by Leon Simon in \cite{Sim21} and can be extend analytically to $\bC\cap B_1\cap\{r=0\}$ in the variables $r^2$ and $y$ and be written as a convergent sum of 
$$\tilde{v}_i=\sum_{j\ge0}v_{ij}(r,y),$$
where $v_{ij}$ is the order $j$ homogeneous (w.r.t. $\rho$) $\beta_i$-harmonic polynomial in the variables $r^2,y$ and $\{v_{ij}\}_{j\ge0}$ restricted to hemisphere are $L^2(\omega_1^{1+\beta_i}d\omega)$-orthogonal. Hence, $v$ admits an expansion
\begin{equation*}
v=\sum_{i,j\ge0}r^{\mu_i}v_{ij}(r,y)\varphi_i(\theta)
\end{equation*}
and for any $0<s\le1$
\begin{equation*}
I_v(s)=s^{-m}\int_{\bC\cap B_{s}}v^2=\sum_{i,j\ge0}c_{ij}^2s^{2(\mu_i+j)}=\sum_{i\ge0}c_i^2s^{2q_i},
\end{equation*}
where $\{q_i\}_{i\ge0}$ is a rearrangement of $\{\mu_i+j\}_{i,j\ge0}$ such that
\begin{equation*}
0=q_0<q_1<q_2<\cdots<q_i\to\infty.
\end{equation*}

The following lemma implies polynomial growth harmonic functions on $\bC$ can be extend analytically to $\sing\bC.$ For general minimal cone, such a property does not hold. For example, let $Y$ be the stationary $1$-dimensional cone consisting of three rays from origin and mutually separated by $120^{\circ}.$ Let $u$ be a function that takes values $0,1,2$ on the three connected components of $Y\backslash\{0\},$ respectively. It is obvious that $u\in\mathcal{H}_1(Y)$ but cannot extend analytically to $\{0\}.$

\begin{lem}
Let $v$ be a harmonic function on $\bC\cap B_1$ with $\int_{\bC\cap B_1}v^2<\infty,$ then $v$ extends analytically to $\bC\cap B_1\cap\{r=0\}.$
\end{lem}
\begin{proof}
For $0<s<1,$ by mean value inequality about subharmonic function
$$\sup_{\bC\cap B_{s}}v^2\le C(n,s)\int_{\bC\cap B_1}v^2,$$ and then
$$\int_{\bC\cap B_{s}}r^{-2}v^2\le C(\bC)\int_{0}^{1/2}\rho^{k-3}d\rho \sup_{\bC\cap B_{s}}v^2< \infty.$$
By the above discuss, $v$ has an expansion on $\bC\cap B_{s}$ and can extend analytically to $\bC\cap B_s\cap\{r=0\}$. Note that $v$ is a smooth function on $\bC\cap B_1\backslash\{r=0\}$ and the arbitrariness of $s$ we infer that $v$ extends analytically to $\bC\cap B_1\cap\{r=0\}.$
\end{proof}

\section{Proof of Theorem \ref{thm11}}
First we concern convergence of harmonic functions under the converge of minimal submanifolds. It is inspired by \cite{Guo24} where focuses on Green's functions.

\begin{lem}\label{lem41}
If $\{u_i\}$ is a sequence of harmonic functions on $M_i$ with $I_{u_i}(s_i)\le 1$ for some positive numbers $s_i\to1,$  then there exists a harmonic function $v$ on $\bC\cap B_1$ which admits an analytic continuation to $\bC\cap B_1\cap\{r=0\}$ and up to a subsequence $u_i$ smoothly converges to $v$ on any compact subset of $\bC\cap B_1\backslash\{r=0\}.$
\end{lem}

\begin{proof}
Fixed $0<s<1,$ form Cacciopolli inequality of harmonic functions, for any $s_i>s,$
$$\int_{M_i\cap B_s}(u_i^2+|\nabla u_i|^2)\le \int_{M_i\cap B_{s_i}} u_i^2+\frac{C}{(s_i-s)^2}\int_{M_i\cap B_{s_i}}u_i^2\le \frac{Cs_i^n}{(s_i-s)^2}.$$
Since $\lim_{i\to\infty}s_i=1,$ $\Vert u_i\Vert_{W^{1,2}(M_i\cap B_s)}$ are uniformly bounded provided $i$ sufficiently large. By Theorem \ref{lem22}, up to a subsequence of $u_i,$  there exists $v^{(s)}\in L^2(\bC\cap B_s)$ such that $u_{i}$ converges to $v^{(s)}$ strongly in $L^2$ on $B_s.$  After taking diagonal subsequence, there exists subsequence of $u_i$
(still denote by $u_i$ ) and $v\in L^2_{loc}(\bC\cap B_1)$ so that $u_i\to v$ in $L^2(B_s)$ for every $0<s<1.$ The harmonicity of $v$ is a natural consequence of local $L^2$ convergence and by $I_{u_i}(s_i)\le1$ we infer that for any $\varepsilon>0,$
$$I_v(1-\varepsilon)=\lim_{i\to\infty}I_{u_i}(1-\varepsilon)\le 1,$$
so $v\in L^2(B_1)$ with $I_v(1)\le1.$ From Theorem 3.1, $v$ can analytic continuation to $\bC\cap B_1\cap\{r=0\}.$ 

By Allard's regularity theorem, $M_i$ smoothly converges to $\bC$ on any compact subset of $B_1\backslash \{r=0\}.$
Hence $u_i$ smoothly converge to $v$ on any compact subset of $\bC\cap B_1\backslash\{r=0\}.$ 
\end{proof}

Before proving a three-circle type theorem as in \cite{Xu16}, we prove an auxiliary lemma which will be used frequently in later.

\begin{lem}\label{lem42}
If $\{a_n\},\{b_n\}$ are two sequences of positive numbers tend to infinity, then there exists subsequences $\{a_{n_k}\}$ of $\{a_n\}$ and $\{b_{m_k}\}$ of $\{b_n\}$ so that $$\lim_{k\to\infty}\frac{a_{n_k}}{b_{m_k}}=\alpha\in[1,2].$$
\end{lem}
\begin{proof}
Set $A_k=\{i:\frac{1}{2}a_k\le b_i\le a_k\}.$ If $\#\{k|A_k\not=\varnothing\}<\infty,$ take $\tilde{k}=\max\{k|A_k\not=\varnothing\},$ then $b_i\le a_{\tilde{k}}$ for all $i.$ It makes a contradicts with $b_i\to\infty.$ It follows that there exists infinitely many indices $k$ such that $a_k/b_{i_k}\in[1,2]$ for $i_k\in A_k.$
 From Bolzano-Weierstrass theorem, one may take subsequences $\{a_{n_k}\} $ and $\{b_{m_k}\}$ so that $$\lim_{k\to\infty}\frac{a_{n_k}}{b_{m_k}}=\alpha\in[1,2].$$
\end{proof}

It is worth noting that $s_0$ in the following three-circle type theorem is independent of the choice of $u\in\mathcal{H}_d(M).$ 

\begin{lem}\label{lem44}
For each $d\not\in\mathscr{D}(\bC),$ there exists a constant $s_0=s_0(M,\bC,d)$ such that the following  property holds. If $u$ is a harmonic function on $M$ and $s\ge s_0,$ then
$$I_u(s)\le 2^{2d}I_u(s/2)\implies I_u(s/2)\le 2^{2d}I_u(s/4).$$ 
\end{lem}
\begin{proof}
Suppose the Lemma failed. There exists a sequence $s_i\to\infty$  and  a sequence of harmonic functions $u_i$ on $M$ so that 
$$I_{u_i}(s_i)\le 2^{2d}I_{u_i}(s_i/2)\text{ but }I_{u_i}(s_i/2)>2^{2d}I_u(s_i/4).$$ 
From Lemma \ref{lem42}, there are subsequences of $\{s_i\}$ and $\{R_i\}$ (still denote by $\{s_i\}$ and $\{R_i\}$) satisfies 
$$\lim_{i\to\infty}\frac{s_i}{R_i}=\alpha\in[1,2].$$

Obviously, $I_{u_i}(s_i/2)\not=0.$ Now  we consider function $\tilde{u}_i$ on $M_i$ given by
$$\tilde{u}_i(x)=\frac{u_i(R_ix)}{\sqrt{I_{u_i}(s_i/2)}}.$$ 
It follows that $\tilde{u}_i$ is harmonic with 
$$I_{\tilde{u}_i}(\frac{s_i}{R_i})\le 2^{2d},I_{\tilde{u}_i}(\frac{s_i}{2R_i})=1\text{ and }I_{\tilde{u}_i}(\frac{s_i}{4R_i})<2^{-2d}.$$

From Lemma \ref{lem41}, there exists a non-zero harmonic function $v$ on $\bC\cap B_{\alpha}$ with
$I_v(s)=\sum_{i\ge0}c_i^2s^{2q_i}$ for any $0<s\le\alpha$
and $$I_v(\alpha)\le 2^{2d},I_v(\frac{\alpha}{2})=1, I_v(\frac{\alpha}{4})\le2^{-2d}.$$
However, Cauchy inequality means that
$$1=(\sum_{i\ge0} c_i^2\alpha^{2q_i}2^{-2q_i})^2\le(\sum_{i\ge0} c_i^2\alpha^{2q_i})(\sum_{i\ge0} c_i^2\alpha^{2q_i}4^{-2q_i})\le1.$$
So exactly one of the $c_i$ is non-zero, and then $d\in\mathscr{D}(\bC).$  It makes a contradiction.
\end{proof}

\begin{cor}\label{cor44}
For any $d\ge0$ and $\delta>0,$ there are constants $s_1=s_1(\bC,d,\delta)\ge1$ and $\Omega_0=\Omega_0(\bC,d,\delta)$ such that the following property holds. If $\Omega\ge\Omega_0$ and  $v$ is a harmonic function on $\bC$ with $I_v(1)=1$ and $I_v(\Omega)\le \Omega^{2d},$ then 
$$I_v(s)\le 2^{2d+\delta}I_v(s/2),\forall s\in[s_1,\log\Omega].$$ 
\end{cor}   
\begin{proof}
Without loss of generality, we assume $d+\delta/2\not\in\mathscr{D}(\bC).$ Observe that minimal submanifold $M$ in Lemma \ref{lem44} can be substituted by minimal cone $\bC$, which ensures the existence of $s_1=s_1(\bC,d,\delta)\ge1$ such that for any $s\ge s_1$ if $I_v(s)>2^{2d+\delta}I_v(s/2),$ then must be hold
$$I_v(2s)>2^{2d+\delta}I_u(s).$$ 

Assume, for contradiction, that the lemma fails. Then there exists a sequence $\Omega_i\to\infty$ and a sequence of harmonic functions $v_i$ on $\bC\cap B_{\Omega_i}$ with $I_{v_i}(1)=1$ and $I_{v_i}(\Omega_i)\le \Omega_i^{2d}$ but
$$I_{v_i}(r_i)> 2^{2d+\delta}I_{v_i}(r_i/2)$$
for some $r_i\in[s_1,\log\Omega_i].$ Let $k\in \mathbb{N}$ satisfies $2^{k}r_i<\Omega_i\le2^{k+1}r_i.$ By iteration, 
$$\Omega_i^{2d}\ge I_{v_i}(2^{k}r_i)\ge 2^{(2d+\delta)k}I_{v_i}(r_i)\ge (\frac{\Omega_i}{2r_i})^{2d+\delta}I_{v_i}(1)\ge\left(\frac{\Omega_i}{2\log\Omega_i}\right)^{2d+\delta}.$$
It makes a contradiction if take $\Omega_i$ enough large.
\end{proof}

For polynomial growth harmonic functions, we have the following three-circle type theorem, which is analogous to Proposition 4.2 in \cite{Hua23}.

\begin{lem}\label{lem45}
For any $d\ge0,$ and $\delta>0,$ there exists $C=C(d,\delta)$ and $s_2=s_2(M,\bC,d,\delta)$ satisfy the following property. If $u$ is a function in $\mathcal{H}_d(M),$ then for any $s\ge t\ge s_2$ it holds
$$I_u(s)\le C(\frac{s}{t})^{2(d+\delta)}I_u(t).$$
\end{lem}
\begin{proof}
Suppose $d+\delta\not\in \mathscr{D}(\bC)$ and $u\not\equiv0.$ Take $\Omega_0$ and $s_1$ as in Corollary \ref{cor44}. By virtue of Lemma \ref{lem21}, we can find a sequence $m_i\to\infty$ such that
$$I_u(\Omega_0^{m_i+1})\le\Omega_0^{2d+\delta}I_u(\Omega_0^{m_i}).$$
Denote by $s_i=\Omega_0^{m_i}.$ By lemma \ref{lem42}, after take subsequence, we can assume that $\lim_{i\to\infty}\frac{s_i}{R_i}=\alpha\in[1,2].$ Let
$$u_i(x)=\frac{u(R_ix)}{\sqrt{I_u(s_i)}},$$
then $I_{u_i}(\frac{s_i}{R_i})=1$ and for any $s\le\frac{s_i}{R_i}\Omega_0,$ 
$$I_{u_i}(s)=\frac{I_u(R_is)}{I_u(s_i)}\le\frac{I_u(\Omega_0^{m_i+1})}{I_u(\Omega_0^{m_i})}\le\Omega_0^{2d+\delta}.$$
From Lemma \ref{lem41}, up to a subsequence, there exists a harmonic function $v$ on $\bC\cap B_{\alpha\Omega_0}$ such that $u_i$ converges locally and smoothly to $v$ and $I_v(\alpha)=1$ and $I_v(\alpha\Omega_0)\le\Omega_0^{2d+\delta}.$ From Corollary \ref{cor44},  then for every $s_1\le s\le\log\Omega_0,$
$$I_v(s)\le 2^{2d+\delta}I_v(s/2).$$
Note that for any $s_1\le s\le\log\Omega_0$ we have the uniformly converge of 
$$\lim_{i\to\infty}\frac{I_{u}(R_is)}{I_{u}(R_is/2)}=\lim_{i\to\infty}\frac{I_{u_i}(s)}{I_{u_i}(s/2)}=\frac{I_{v}(s)}{I_{v}(s/2)}\le 2^{2d+\delta}.$$
So there exists large $i_0$ such that $I_u(s)\le \Omega^{2(d+\delta)}I_u(s/2) $ for $R_is_1\le s\le R_i\log\Omega_0$ when $i\ge i_0.$ By Lemma \ref{lem44}, there is $s_2(M,\bC,d+\delta)>0$ so that $I_u(s)\le \Omega^{2(d+\delta)}I_u(s/2) $ for $s_2\le s\le R_i\log\Omega_0$ when $i\ge i_0.$ Since $R_i\to\infty,$  
$$I_u(s)\le 2^{2(d+\delta)}I_u(s/2),\forall s\ge s_2.$$
Furthermore, if $s\ge t\ge s_2,$ there is unique $k\in\mathbb{N}$ such that $2^kt\le s<2^{k+1}t,$ then
$$\frac{I_u(s)}{I_u(t)}\le\frac{I_u(2^{k+1}t)}{I_u(2^kt)}\frac{I_u(2^{k}t)}{I_u(2^{k-1}t)}\cdots\frac{I_u(2t)}{I_u(t)}\le 2^{2(k+1)(d+\delta)}\le 2^{2(2+2\delta)}(\frac{s}{t})^{2(d+\delta)}.$$
Set $C=2^{2d+4\delta}$ the proof is complete.
\end{proof}

\begin{rk}\label{rk36}
Under the hypothesis of Lemma \ref{lem45}, take $R\ge s_2$ and let $\tilde{M}=R^{-1}M.$ Then for any $u\in\mathcal{H}_d(\tilde{M})$ and $s\ge1,$
$$I_{u}(s)=(Rs)^{-n}\int_{M\cap B_{Rr}}u^2(x/R)\le Cs^{2(d+\delta)} R^{-n}\int_{M\cap B_R}u^2(x/R)= Cs^{2(d+\delta)}I_u(1).$$
\end{rk}

The proof of Theorem \ref{thm11} is mainly modified from that of Proposition 4.4 in \cite{Hua19} and Theorem 1.6 in \cite{Hua23}.

\begin{proof}[Proof of Theorem \ref{thm11}]
First, we claim that for any $d>0,$ $h_d(M)\le h_d(\bC).$ Given any $d\ge0,$ choose small $\delta>0$ so that $(d,d+\delta)\cap\mathscr{D}(\bC)=\varnothing.$  Let $m=h_d(M)-1.$ Take $s_2$ as in Lemma \ref{lem45}, we may assume $R_i\ge s_2.$  Suppose $\{u_i^{(a)}\}_{1\le a\le m}$ is a sequence linearly independent non-constant functions in $\mathcal{H}_d(M_i)$ with respect to inner product $J_1$ such that 
\begin{equation*}
J_1(u_i^{(a)},1)=0,J_1(u_i^{(a)},u_i^{(b)})=\delta_{ab}
\end{equation*}
    
By Remark \ref{rk36},  $I_{u_{i}^{(a)}}(s)\le Cs^{2(d+\delta)}$ when $s\ge1.$ Then the mean value inequality of subharmonic function implies that
\begin{equation*}
\sup_{M_i\cap B_s}|u_i^{(a)}|^2\le C(n)I_{v_{i,a}}(2s)\le C(n,d,\delta)s^{2(d+\delta)}.
\end{equation*}
    
From Lemma \ref{lem41}, we can find harmonic functions $v^{(a)}$ on $\bC$  satisfy
\begin{equation*}
J_1(v^{(a)},1)=0,J_1(v^{(a)},v^{(b)})=\delta_{ab}\ \text{and}\ \sup_{\bC\cap B_s}|v^{(a)}|^2\le C(n,d,\delta)s^{2(d+\delta)}.
\end{equation*}   
It means that $v^{(a)}\in\mathcal{H}_{d+\delta}(\bC)=\mathcal{H}_d(\bC)$ are non-constant and linearly independent. Thus, $h_d(M)=m+1\le h_d(\bC).$

To complete the proof, we show the reverse inequality: $h_d(M)\ge h_d(\bC)$ for any positive $d\not\in\mathscr{D}(\bC).$ Fixed $\delta>0$ and $p_0\in M_i,$ let $p_i=\frac{1}{R_i}p_0\in M_i$ and $S_{\lambda}=\{x\in \R^N:d(x,\sing\bC)<\delta\}.$ By virtue of Lemma 2.7 in \cite{Guo24}, there there are $\lambda_i$-Gromov-Hausdorff approximations $\phi_i:(M_i,p_i)\to (\bC,0)$ so that $\lambda_i\to0$ as $i\to\infty$ and $\phi_i:\bC\backslash S_{\lambda}\to M_i\backslash S_{\lambda}$ is a diffeomorphism which converges smoothly to identity map on $\bC\backslash S_{\lambda}.$

Let $v$ be a non-zero function in $\mathcal{H}_d(\bC)$ with $v(0)=0$ and $U_{i,\lambda}$ be an smooth exhaustion of $B_1\backslash S_{\lambda}.$ Let $u_{i,\lambda}$ be the solution of Dirichlet problem
\begin{equation*}
    \begin{cases}
    \Delta u_{i,\lambda}=0,& M_i\cap U_{i,\lambda}\\
    u_{i,\lambda}=v\circ\phi_i^{-1},& M_i\cap\p U_{i,\lambda}
    \end{cases}
\end{equation*}
Since $M_i$ converges smoothly to $\bC$ outside $S_{\lambda},$ there exists harmonic function $v_{\lambda}$ in $\bC\cap U_{i,\lambda}$ so that up to a subsequence, $u_{i,\lambda}$ converges to $v_{\lambda}$ smoothly. From maximum principle, we have $v_{\lambda}=v$ in $\bC\cap U_{i,\lambda}.$ On the other hand, from Rellich compactness theorem, there is a harmonic function $\bar{u}_i$ on $M_i\cap B_1$ such that after take a subsequence, $u_{i,\lambda}$ converges to $\bar{u}_i$ uniformly on compact subsets of $M_i\cap B_1$ as $\lambda\to0.$ Hence, using Lemma \ref{lem41} and taking diagonal sequence, there is a subsequence of $\bar{u}_i$ (still denote by $\bar{u}_i$) converges to $v$ in $C_{loc}^{\infty}(B_1).$ 

Take $\delta>0$ so that $d+\delta\not\in\mathscr{D}(\bC)$ and $s_0(M,\bC,d+\delta)$ as in Lemma \ref{lem44}. Let $u_i(x)=\bar{u}_i(x)-\bar{u}_i(p_i)$ and
$$\tilde{u}_i(x)=\frac{u_i(x/R_i)}{\sqrt{I_{u_i}(s_0/R_i)}}.$$    
Then $\tilde{u}_i$ is harmonic on $M\cap B_{R_i}$ with $\tilde{u}_i(p_0)=0$ and $I_{\tilde{u}_i}(s_0)=1.$  

Note that for any $s\in[\frac{1}{4},\frac{3}{4}],$ we have the following uniform convergence
$$\lim_{i\to\infty}\frac{I_{\tilde{u}_i}(R_is)}{I_{\tilde{u}_i}(R_is/2)}=\lim_{i\to\infty}\frac{I_{u_i}(s)}{I_{u_i}(s/2)}=\frac{I_v(s)}{I_v(s/2)}\le 2^{2d}.$$
Thus, there exists $i_0$ independent of $s$ so that $I_{\tilde{u}_i}(s)\le 2^{2(d+\delta)}I_{\tilde{u}_i}(s/2)$ for any $s\in [\frac{1}{4}R_i,\frac{3}{4}R_i]$ when $i\ge i_0.$ By Lemma \ref{lem44}, for all $i\ge i_0$ and $s_0\le s\le \frac{3}{4}R_i,$
$$I_{\tilde{u}_i}(s)\le 2^{2(d+\delta)}I_{\tilde{u}_i}(\frac{s}{2}),\forall s\in[s_0,\frac{3}{4}R_i],$$ 
and then by iteration, we have 
$$I_{\tilde{u}_i}(s)\le  C(M,d,\delta)s^{2d+2\delta}.$$ 
From mean value inequality,
$$\tilde{u}_i^2(p)\le C(n)I_{\tilde{u}_i}(2|p|)\le  C(M,n,d,\delta)(\max\{s_0,|p|\})^{2(d+\delta)},\forall p\in M\cap B_{\frac{3}{8}R_i}.$$

By Rellich compactness theorem and taking diagonal subsequence, there exists subsequence of $\tilde{u}_i$ converge to a smooth harmonic function $u$ defined $M$  with $u(p_0)=0$ and $I_{u}(s_0)=1$ and
$$u^2(p)\le C(1+|p|^{2(d+\delta)}).$$
It follows that $u_i$ is a non-constant function in $\mathcal{H}_{d+\delta}(M).$
    
For $m\ge2,$ let $v^{(1)},v^{(2)},\cdots, v^{(m)}$ be a family linear independent non-constant functions in $\mathcal{H}_d(\bC)$ with $v^{(a)}(0)=0$ for all $1\le a\le m.$ Now for $1\le a,b\le m-1,$ from the above construction we can suppose the following functions exist.
    
(1) non-constant functions $u_i^{(a)}$ on $M_i\cap B_1$ with $u_i^{(a)}(p_i)=0$ and converge to $v^{(a)}$ locally and smoothly. 
    
(2) non-constant functions $w_i^{(a)}$ on $M_i\cap B_1$ given by
$$w_i^{(1)}=u_i^{(1)},w_i^{(a)}=u_i^{(a)}+\sum_{b<a}\alpha_i^{(a,b)}w_i^{(b)},a>1,$$
where $\alpha_i^{(a,b)}$ determined by $J_{s_0/R_i}(w_i^{(a)},w_i^{(b)})=0$ for $1\le a\not=b\le m-1.$
    
(3) non-constant functions $\tilde{w}^{(a)}_i$ on $M\cap B_{R_i}$ defined by 
$$\tilde{w}^{(a)}_i(x)=\frac{w_i^{(a)}(x/R_i)}{\sqrt{I_{w^{(a)}_i}(s_0/R_i)}},$$
and $\tilde{w}_i^{(a)}$ converges to non-constant functions $w^{(a)}\in\mathcal{H}_{d+\delta}(M)$ with $J_{s_0}(w^{(a)},w^{(b)})=\delta_{ab}.$

For $a=m,$ there are non-constant harmonic functions ${u_i^{(m)}}$ on $M_i\cap B_{1}$ with $u_i^{(m)}(p_i)=0$ and converges locally and smoothly to $v^{(m)}.$ Let $\{\alpha_i^{(m,b)}\}_{1\le b\le m-1}$ be a sequence numbers so that $w_i^{(m)}=u_i^{(m)}+\sum_{1\le b<m}\alpha_i^{(m,b)}u_i^{(b)}$ satisfies $$J_{s_0/R_i}(w_i^{(m)},w_i^{(b)})=0,\text{ when } b\not=m.$$ After take a subsequence, we may assume $\lim_{i\to\infty}\alpha_i^{(m,b)}=\beta^{(m,b)}\in\R\cup\{\infty\}.$ We now consider two cases depending on whether $\beta^{(m,b)}$ is bounded.
    
Case 1. If all $\beta^{(m,b)}$ are bounded. Then $w_i^{(m)}$ converges to a harmonic function $v^{(m)}+\sum_{1\le b<m}\beta^{(m,b)}v^{(b)}$ which is non-zero because of the linear independence of $\{v^{(a)}\}_{1\le a\le m}.$ As the above construction, the sequence
$$\tilde{w}_i^{(m)}(x)=\frac{w_i^{(m)}(x/R_i)}{\sqrt{I_{w_i^{(m)}}(s_0/R_i)}}$$ (up to a subsequence) converges to a non-constant function $w^{(m)}\in\mathcal{H}_{d+\delta}(M)$ and for every $1\le b<m$
$$J_{s_0}(w^{(m)},w^{(b)})=\lim_{i\to\infty}J_{s_0}(\tilde{w}_i^{(m)},\tilde{w}_i^{(b)})=\lim_{i\to\infty}\frac{J_{s_0/R_i}(w_i^{(m)},w_i^{(b)})}{\sqrt{I_{w_i^{(m)}}(s_0/R_i)}\sqrt{I_{w_i^{(b)}}(s_0/R_i)}}=0.$$
It means $\{w^{(a)}\}_{1\le a\le m}$ are linear independent.
    
Case 2. Otherwise, up to a subsequence, there is one sequence $\alpha_i^{(m,b_0)}$ so that $\lim_{i\to\infty}\alpha_i^{(m,b_0)}=\infty$ and $\lim_{i\to\infty}\frac{\alpha_i^{(m,b)}}{\alpha_i^{(m,b_0)}}=\gamma^{(m,b)}\in\R$. Let
$$\hat{w}_i^{(m)}=\frac{1}{\alpha_i^{(m,b_0)}}w_i^{(m)},$$
it still holds $J_{s_0/R_i}(\hat{w}_i^{(m)},w_i^{b})=0$ when $b\not=m$ and $\hat{w}_i^{(m)}$ converges to  non-zero harmonic function $v^{(b_0)}+\sum_{1\le b\not=b_0<m}\gamma^{(m,b)}v^{(b)}.$ Denote by 
$$\tilde{w}_i^{(m)}=\frac{\hat{w}_i^{(m)}(x/R_i)}{\sqrt{I_{\hat{w}_i^{(m)}}(s_0/R_i)}}.$$
As in case 1, after take a subsequence, $\tilde{w}_i^{(m)}$ converges to a non-constant function $w^{(m)}\in\mathcal{H}_{d+\delta}(M)$ with $J_{s_0}(w^{(m)},w^{(b)})=0$ for $b\not=m.$ 

The above discuss implies $h_{d+\delta}(M)\ge h_d(\bC)$ provided $d+\delta\not\in\mathscr{D}(\bC).$ In particular, if $d\not\in\mathscr{D}(\bC),$ one may choose small $\delta>0$ so that $[d-\delta,d]\cap\mathscr{D}(\bC)=\varnothing,$ then $h_{d}(M)\ge h_{d-\delta}(\bC)=h_d(\bC).$

\end{proof}

\end{document}